\documentclass[11pt, reqno]{amsart}

\allowdisplaybreaks

\usepackage{graphicx}
\usepackage{color}
\usepackage{float}
\usepackage{tikz}
\usetikzlibrary{positioning}
\usetikzlibrary{calc}

\newtheorem{theorem}{Theorem} 
\newtheorem{lemma}[theorem]{Lemma}

\newtheorem{corollary}[theorem]{Corollary}

\theoremstyle{definition}

\usepackage{amsmath}
\usepackage{amssymb}
\usepackage{tikz}
\usepackage{amsthm}
\usepackage{amscd}
\usepackage{enumerate}
\usepackage{amsfonts}
\usepackage{fancyhdr}
\usepackage{mathtools}
\usepackage{framed}
\usepackage[mathscr]{eucal}

\def\C{{\mathbb C}}

\def\Z{{\mathbb Z}}

\renewcommand{\leq}{\leqslant}
\renewcommand{\geq}{\geqslant}

\usepackage{todonotes}

\begin{document}
\title[Extending the Affirmative Action Problem]{Extending the Affirmative Action Problem: mixing numbers and integrated colorings of graphs}

\author[Burnette]{Charles Burnette\textsuperscript{*}}
            \thanks{\textsuperscript{*}Funded by a Title III grant through Xavier University of Louisiana.}
           \address{Department of Mathematics, Xavier University of Louisiana,New Orleans, Louisiana 70125-1098, USA}
            \email{cburnet2@xula.edu}           
\author[Caton]{Broden Caton\textsuperscript{\textdagger}}
            \thanks{\textsuperscript{\textdagger}This research was supported by the Center for Undergraduate Research and Graduate Opportunity at XULA}
          \address{Department of Physics and Computer Science, Xavier University of Louisiana, New Orleans, Louisiana 70125-1098, USA}
           \email{bcaton@xula.edu}           
\author[Coward]{Olivia Coward\textsuperscript{\textdagger}}
            \address{Department of Mathematics, Xavier University of Louisiana, New Orleans, Louisiana 70125-1098, USA}
             \email{ocoward@xula.edu}
\author[Davis]{Julian Davis\textsuperscript{\textdagger}}
            \address{Department of Mathematics, Xavier University of Louisiana, New Orleans, Louisiana 70125-1098, USA}
             \email{jmdavis4@xula.edu}
\author[Teter]{Austin Teter\textsuperscript{\textdagger}}
            \address{Department of Mathematics, Xavier University of Louisiana, New Orleans, Louisiana 70125-1098, USA}
             \email{ateter@xula.edu}


            
\begin{abstract}
Consider a graph whose vertices are colored in one of two colors, say black or white. A white vertex is called integrated if it has at least as many black neighbors as white neighbors, and similarly for a black vertex. The coloring as a whole is integrated if every vertex is integrated. A classic exercise in graph theory, known as the Affirmative Action Problem, is to prove that every finite simple graph admits an integrated coloring. The solution can be neatly summarized with the one-liner: ``maximize the number of balanced edges,'' that is, the edges that connect neighbors of different colors. However, not all integrated colorings advertise the maximum possible number of balanced edges.

In this paper, we characterize and enumerate the integrated colorings for complete graphs, bicliques, paths, and cycles. We also derive the distributions and extremal values for the mixing numbers (the number of balanced edges) across all integrated colorings over these families of graphs. For paths and cycles in particular, we use the quasi-powers framework for probability generating functions to prove a central limit theorem for the mixing number of random integrated colorings. Lastly, we obtain an upper bound for the number of integrated colorings over any fixed graph via the second-moment method. A specialized bound for the number of integrated colorings over a regular graph is also obtained.
\end{abstract}           
\maketitle

\section{Introduction}

By a \textit{coloring} of a (labeled) graph $G = (V, E),$ where $V$ and $E$ are the sets of vertices and edges, respectively, of $G,$ we always mean a partition of $V$ into two disjoint subsets, i.e.\! a map $\mathcal{C}: V \rightarrow \{\text{black}, \text{white}\}.$ Given a coloring $\mathcal{C}$ of $G,$ let us call an edge \textit{balanced} if it connects two differently colored vertices. Edges that connect vertices of the same color are \textit{unbalanced}. Define the \textit{mixing number} of $\mathcal{C},$ denoted by $\text{mix}(\mathcal{C}),$ to be the number of balanced edges. For each vertex $v \in V,$ the \text{mixing number} of $v$, denoted $\text{mix}(v),$ is the number of opposite colored vertices adjacent to $v,$ that is, the total number of balanced edges incident to $v$.

We say that a vertex $v$ is \textit{integrated} if $\text{mix}(v) \geq \frac{1}{2}\deg(v).$ The coloring $\mathcal{C}$ as a whole is \textit{integrated} if every vertex is integrated. Integrated colorings are more widely referred to as unfriendly partitions in the literature, as first published in \cite{AMP}. Nonetheless, we adopt the more inclusive language here.

The ``Affirmative Action Problem,'' reportedly posed by Donald Newman (see Example 2.1.9 of \cite{Zeitz}, pp.\! 20-22), asks whether all finite simple graphs admit an integrated coloring. The answer is yes, and the following proof, due to Jim Propp, is quite simple: any coloring $\mathcal{C}$ of a fixed graph with the largest possible mixing number is necessarily integrated. Indeed, if there is a vertex that is not integrated, then switching its color would contradictorily increase $\text{mix}(\mathcal{C}).$

Although integrated colorings are well studied for infinite graphs, little has seemingly been discussed about them in finite settings. Cowen \cite{Cowen} relates a local search algorithm for producing an integrated coloring of a graph to the \textsc{Max-Cut} problem, which concerns partitioning the vertex set $V$ into two complementary sets such that the number of edges between the two sets is as large as possible. Propp's proof therefore informs us that a sufficient condition for a coloring to be integrated is for its mixing number to equal the size of the max cut. The converse of this does not always hold. For example, consider the two integrated colorings of the square circuit graph provided in Figure \ref{squarecircuit} below. The coloring on the left has the maximum possible mixing number of 4 whereas the one on the right only has a mixing number of 2.

\begin{figure}[h]
\caption{Example of a graph that has integrated colorings with different mixing numbers.}\label{squarecircuit}
            \begin{minipage}{0.45\textwidth}
            \centering
            \begin{tikzpicture}[blacknode/.style={circle, draw=black, fill=black, very thick, minimum size=2}, whitenode/.style={circle, draw=black, fill=white, very thick, minimum size=2},]
                \node[label=above:{1}][blacknode] (black1) at (0,0){};
                \node[label=above:{2}][whitenode] (black2) at (1,0){};
                \node[label=below:{4}][whitenode] (black3) at (0,-1){};
                \node[label=below:{3}][blacknode] (black4) at (1,-1){};
                \draw[-] (black1) -- (black2);
                \draw[-] (black3) -- (black4);
                \draw[-] (black1) -- (black3);
                \draw[-] (black2) -- (black4);
            \end{tikzpicture}
            \end{minipage}
            \begin{minipage}{0.45\textwidth}
            \centering
            \begin{tikzpicture}[blacknode/.style={circle, draw=black, fill=black, very thick, minimum size=2}, whitenode/.style={circle, draw=black, fill=white, very thick, minimum size=2},]

            \node[label=above:{1}][whitenode] (white1) at (0,0){};
            \node[label=above:{2}][whitenode] (white2) at (1,0){};
            \node[label=below:{4}][blacknode] (black1) at (0,-1){};
            \node[label=below:{3}][blacknode] (black2) at (1,-1){};
            \draw[-] (white1) -- (white2);
            \draw[-] (white2) -- (black2);
            \draw[-] (white1) -- (black1);
            \draw[-] (black1) -- (black2);
            \end{tikzpicture}
            \end{minipage}
\end{figure}
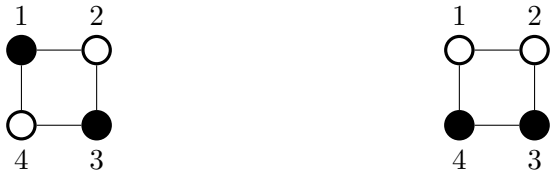

A natural extension of the Affirmative Action Problem is to then characterize and enumerate every integrated coloring for a given graph, as well as to catalog their corresponding mixing numbers. Let $\text{IC}(G)$ be the set of all integrated colorings of $G$ and let $\text{ic}(G)$ be the cardinality of $\text{IC}(G).$ We define the \textit{integrated mixture spectrum} of $G$ as $\text{ims}(G) := \{\text{mix}(\mathcal{C}): \mathcal{C} \in \text{IC}(G)\}.$ Finally, let $\text{ims}^-(G) := \min(\text{ims}(G))$ and $\text{ims}^+(G) := \max(\text{ims}(G)).$ Since the sum of the vertex mixing numbers double counts the number of balanced edges and, in any integrated coloring of $G,$
\begin{equation}
\sum_{v \in V} \text{mix}(v) \geq\sum_{v \in V} \frac{1}{2}\deg(v) = |E|,
\end{equation}
it follows that $\text{ims}^-(G) \geq \frac{1}{2}|E|.$ Various lower bounds for $\text{ims}^+(G)$ can be attained from prior work on the \textsc{Max-Cut} problem, including perhaps most notably Edwards's \textsc{Max-Cut} bound \cite{Edwards2}
\begin{equation}
    \text{ims}^+(G) \geq \left\lceil\frac{|E|}{2} + \sqrt{\frac{|E|}{8} + \frac{1}{64}} - \frac{1}{8} \right\rceil
\end{equation}
and the Edwards-Erd\H{o}s bound for connected graphs \cite{Edwards1}
\begin{equation}
    \text{ims}^+(G) \geq \frac{|E|}{2} + \frac{|V|-1}{4}.
\end{equation}

In this paper, we find $\text{IC}(G),$ $\text{ims}(G),$ and the exact distribution of $\text{mix}(\mathcal{C})$ for complete graphs, bicliques, paths, and cycles. To help readers navigate through this material, Table \ref{summary} lists our results within each of these graph families. For paths and cycles in particular, we derive bivariate rational generating functions for random integrated colorings and use them to acquire probability generating functions for $\text{mix}(\mathcal{C}).$ Since these generating functions belong to the algebraic class of meromorphic functions with a single singularity on their circles of convergence, the principles of analytic combinatorics suggest that the distribution of $\text{mix}(\mathcal{C})$ has a normal limit law as the path/cycle grows. These details are carried out towards the ends of Sections 3 and 4. We conclude by implementing the second-moment method to extract an upper bound for $\text{ic}(G)$ in terms of the order and degree sequence of $G$, which is then specialized into a bound for $\text{ic}(G)$ on regular graphs.
\begin{table}[H]
\centering
\caption{Integrated coloring enumerations for some graphs}\label{summary}
\resizebox{\columnwidth}{!}{%
\begin{tabular}{l | l l l}
\hline 
\textit{Graph G} & $\text{ic}(G)$ & $\text{ims}(G)$ & \textit{Distribution of} mix$(\mathcal{C})$ \\
\hline \hline
\begin{tabular}{l}
Complete: $K_{2n}$ \\
(pp. 3--4) \\
\end{tabular} & $\displaystyle \binom{2n}{n}$ & $\{n^2\}$ & $\text{mix}(\mathcal{C}) \equiv n^2$ \\
\hline
\begin{tabular}{l}
Complete: $K_{2n-1}$ \\
(pp. 3--4) \\
\end{tabular} & $\displaystyle 2\binom{2n-1}{n}$ & $\{n^2-n\}$ & $\text{mix}(\mathcal{C}) \equiv n^2-n$ \\
\hline
\begin{tabular}{l}
Biclique: $K_{m,n},$ \\ $m$ or $n$ odd \\
(pp. 4--5) \\
\end{tabular} & $2$ & $\{mn\}$ & $\text{mix}(\mathcal{C}) \equiv mn$ \\
\hline
\begin{tabular}{l}
Biclique: $K_{m,n},$ \\ $m,$ $n$ both even \\
(pp. 4--5) \\
\end{tabular} & $\displaystyle 2 + \binom{m}{m/2}\binom{n}{n/2}$ & $\{mn/2, mn\}$ & \begin{tabular}{l} \\ two integrated $\mathcal{C}$ with $\text{mix}(\mathcal{C}) = mn;$ \\ the rest have $\text{mix}(\mathcal{C}) = mn/2$ \end{tabular} \\
\hline
\begin{tabular}{l}
Path: $P_n$ \\
(pp. 5-6) \\
\end{tabular} & \begin{tabular}{l}
$2$ for $n = 1,$ \\
$2F_{n-1}$ for $n \geq 2$ \\
\end{tabular} & \begin{tabular}{l}
integers b/w \\
$\frac{n-1}{2}$ and $n-1$ \\
\end{tabular} & \begin{tabular}{l}
for $k \in \text{ims}(P_n),$ \\
$\text{Pr}[\text{mix}(\mathcal{C}) = k] = \frac{2\binom{k-1}{n-k-1}}{ic(P_n)}$ \\
\end{tabular}  \\
\hline
\begin{tabular}{l}
Cycle: $C_n$ \\
(pp. 9--11) \\
\end{tabular} & $L_n + 2\cos(2n\pi/3)$ & \begin{tabular}{l}
even integers b/w \\
$\frac{n}{2}$ and $n$ \\
\end{tabular} & \begin{tabular}{l}
for $2k \in \text{ims}(C_n),$ \\
$\text{Pr}[\text{mix}(\mathcal{C}) = 2k] = \frac{2\binom{2k-1}{n-2k} + 4\binom{2k-1}{n-2k-1}}{\textup{ic}(C_n)}$ \\
\end{tabular}  \\
\hline
\end{tabular}
}
\end{table}

\section{Complete Graphs and Bicliques}

Intuitively, there is an equal number of black and white vertices in any integrated coloring of $K_{2n},$ and they must be as close to equinumerous as possible in integrated colorings of $K_{2n-1}$. One should also suspect that the inherent disparity in the number of black and white vertices in the latter implies that $\text{ic}(K_{2n-1})$ is twice what it would be for $\text{ic}(K_{2n}).$

\begin{theorem}
For each $n \in \mathbb{N},$ $\textup{ic}(K_{2n}) = \binom{2n}{n}$ and $\textup{ic}(K_{2n-1}) = 2\binom{2n-1}{n}.$ Furthermore, $\textup{mix}(\mathcal{C}) = n^2$ for all $\mathcal{C} \in \textup{IC}(K_{2n})$ and $\textup{mix}(\mathcal{C}) = n^2-n$ for all $\mathcal{C} \in \textup{IC}(K_{2n-1}).$
\end{theorem}

\begin{proof}
Given a positive integer $r,$ suppose $\mathcal{C}$ is a coloring of $K_r$ with $x$ black vertices and $y$ white vertices. If $x > \lceil r/2\rceil,$ then
\[\text{mix}(v) = y \leq \left\lfloor \frac{r}{2} \right\rfloor - 1 < \frac{r}{2}-\frac{1}{2}=\frac{1}{2}\text{deg}(v)\]
for every black vertex $v.$ Similarly, if $x < \lfloor r/2 \rfloor,$ then $\text{mix}(w) = x < \frac{1}{2}\text{deg}(w)$ for every white vertex $w.$ Thus, if $\mathcal{C}$ is integrated, then $x = \lfloor r/2 \rfloor$ (and $y = \lceil r/2 \rceil$) or $x = \lceil r/2 \rceil$ (and $y = \lfloor r/2 \rfloor$). When this is the case, every vertex has a mixing number of at least $\lfloor r/2 \rfloor,$ thus ensuring integration.

If $r = 2n,$ then $x = n = y,$ and so $\text{mix}(\mathcal{C}) = n^2$ for all $\mathcal{C} \in IC(K_{2n})$ since there are $n$ choices each for the black and white endpoints of a balanced edge. If instead $r = 2n - 1,$ then $x = n$ (and $y = n - 1$) or $x = n - 1$ (and $y = n$). The same reasoning dictates that $\text{mix}(\mathcal{C}) = n(n-1)$ for all $\mathcal{C} \in \text{IC}(K_{2n-1}).$ Moreover, $\text{ic}(K_{2n}) = \binom{2n}{n}$ because one simply needs to decide which half of the $2n$ vertices are black. On the other hand, $\text{ic}(K_{2n-1}) = 2\binom{2n-1}{n}$ since one also needs to pick which color to impose on exactly $n$ of the vertices.
\end{proof}

For a bipartite graph $G$, coloring the two parts of $G$ differently produces two integrated colorings. If in addition $G$ is a biclique, then having an equal number of black and white vertices in each part yields another possible type of integrated coloring. This suggests that, as in Theorem 1, the parity of $m$ and $n$ must be taken into account to unravel $\text{IC}(K_{m,n}).$

\begin{theorem}
Let $m, n \in \mathbb{N}.$ Then
\[\textup{ic}(K_{m,n}) = \begin{cases}
  2 & \text{if $m$ or $n$ is odd,}  \\
  2+ \binom{m}{m/2}\binom{n}{n/2} & \text{if $m$ and $n$ are even.}
\end{cases}\]
If $m$ or $n$ is odd, then $\textup{mix}(\mathcal{C}) = mn$ for all $\mathcal{C} \in \textup{IC}(K_{m,n}).$ However, if $m$ and $n$ are even, then two of the colorings in $\textup{IC}(K_{m,n})$ have mixing number $mn$ while the rest have mixing number $mn/2.$ 
\end{theorem}

\begin{proof}
Let $V_1$ and $V_2$ be the parts of $K_{m,n}.$ Consider the following two cases that lend themselves to how to reach an integrated coloring of $K_{m,n}.$

\textit{Case 1:} at least one of $V_1$ or $V_2$ is monochromatic. Without loss of generality, assume that every vertex in $V_1$ is black. Then every vertex in $V_2$ is white; otherwise there would be a black vertex in $V_2$ whose only neighbors are the black vertices comprising $V_1.$ This clearly results in an integrated coloring of $K_{m,n}$ where all $mn$ edges are balanced.

\textit{Case 2:} neither $V_1$ nor $V_2$ is monochromatic. For each $i \in \{1, 2\},$ let $x_i$ and $y_i$ be the numbers of black and white vertices, respectively, in $V_i.$ We need $y_1 \geq x_1$ ($x_1 \geq y_1$) in order for the black (white) vertices in $V_2$ to be integrated. Hence $x_1 = y_1.$ Likewise, $x_2 = y_2.$ This yields an integrated coloring of $K_{m,n}$ where half of the edges are balanced.

The two integrated colorings associated with case 1 are the only ones possible for $K_{m, n}$ when $m$ or $n$ is odd. Else, $m$ and $n$ are even, in which case 2 adds an extra $\binom{m}{m/2}\binom{n}{n/2}$ integrated colorings; simply choose which half of the vertices in $V_1$ and which half of the vertices in $V_2$ are black.
\end{proof}

\section{Paths}

Throughout this section and the next, we will use the ``conventional'' vertex-labelings of paths and cycles as shown in Figure \ref{conventionallabels}.

\begin{figure}[h]
\caption{Conventional labelings of $P_n$ ($C_n$): the vertex labels increase by $1\, (1\, \text{mod}\ n)$ as one walks right (clockwise).}\label{conventionallabels}
            \begin{minipage}{0.45\textwidth}
            \centering
            \begin{tikzpicture}[
node/.style={circle, draw=black, fill=black, very thick, minimum size=2},]
                \node[label=above:{1}][node] (node1) at (0,0){};
\node[label=above:{2}][node] (node2) at (1,0){};
\node[label=above:{$n$}][node] (node3) at (3,0){};
\node at ($(node2)!.5!(node3)$) {\ldots};
\draw[-] (node1) -- (node2);
\draw[-] (node2) -- (1.5,0);
\draw[-] (2.5,0) -- (node3);
\node at (1.5,-0.5) {$P_n$};
\end{tikzpicture}
            \end{minipage}
            \begin{minipage}{0.45\textwidth}
            \centering
            \begin{tikzpicture}[node/.style={circle, draw=black, fill=black, very thick, minimum size=2},]

            \node[label=above:{1}][node] (node1) at (0,0.63662){};
            \node[label=right:{2}][node] (node2) at (0.63662,0){};
            \node[label=left:{$n$}][node] (node3) at (-0.63662,0){};
            \node at (0,-0.63662) {\ldots};
            \draw (node1) arc (90:0:0.63662);
            \draw (node2) arc (0:-60:0.63662);
            \draw (node3) arc (180:240:0.63662);
            \draw (node1) arc (90:180:0.63662);
            \node at (0,-1) {$C_n$};
            \end{tikzpicture}
            \end{minipage}
\end{figure}
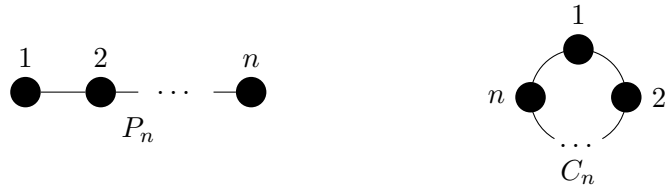

It is clear that $\text{IC}(P_n)$ is in bijection with the set of all $n$-bit binary strings that avoid runs of length 3, the number of which is well-known to satisfy the Fibonacci recurrence
\begin{equation}
    \text{ic}(P_n) = \text{ic}(P_{n-1})+\text{ic}(P_{n - 2}),\ \text{for}\ n \geq 4.
\end{equation}
Taking $F_n$ to denote the $n^{\text{th}}$ Fibonacci number, we have $\text{ic}(P_2) = 2 = 2F_1$ and $\text{ic}(P_3) = 2 = 2F_2.$ (See Figure \ref{coloringpaths}.) It is now straightforward to compute $\text{ic}(P_n)$ in general.

\begin{figure}[h]
\caption{The integrated colorings of $P_1,$ $P_2,$ and $P_3$.}\label{coloringpaths}
\begin{minipage}{0.15\textwidth}
\centering
\begin{tikzpicture}[
blacknode/.style={circle, draw=black, fill=black, very thick, minimum size=2},
whitenode/.style={circle, draw=black, fill=white, very thick, minimum size=2},
]

\node[label=above:{1}][blacknode] (black1) at (0,0){};
\node[label=above:{1}][whitenode] (white1) at (0,1){};
\node at (0,-0.5) {$P_1$};

\end{tikzpicture}
\end{minipage}
\begin{minipage}{0.3\textwidth}
\centering
\begin{tikzpicture}[
blacknode/.style={circle, draw=black, fill=black, very thick, minimum size=2},
whitenode/.style={circle, draw=black, fill=white, very thick, minimum size=2},
]

\node[label=above:{2}][blacknode] (black1) at (1,0){};
\node[label=above:{1}][whitenode] (white1) at (0,0){};
\draw[-] (black1) -- (white1);

\node[label=above:{2}][whitenode] (white1) at (1,1){};
\node[label=above:{1}][blacknode] (black1) at (0,1){};
\draw[-] (black1) -- (white1);

\node at (0.5,-0.5) {$P_2$};

\end{tikzpicture}
\end{minipage}
\begin{minipage}{0.3\textwidth}
\centering
\begin{tikzpicture}[
blacknode/.style={circle, draw=black, fill=black, very thick, minimum size=2},
whitenode/.style={circle, draw=black, fill=white, very thick, minimum size=2},
]

\node[label=above:{1}][blacknode] (black1) at (0,0){};
\node[label=above:{2}][whitenode] (white1) at (1,0){};
\node[label=above:{3}][blacknode] (black2) at (2,0){};
\draw[-] (black1) -- (white1);
\draw[-] (white1) -- (black2);

\node[label=above:{1}][whitenode] (white2) at (0,1){};
\node[label=above:{2}][blacknode] (black3) at (1,1){};
\node[label=above:{3}][whitenode] (white4) at (2,1){};
\draw[-] (white2) -- (black3);
\draw[-] (black3) -- (white4);

\node at (1,-0.5) {$P_3$};
\end{tikzpicture}
\end{minipage}
\end{figure}
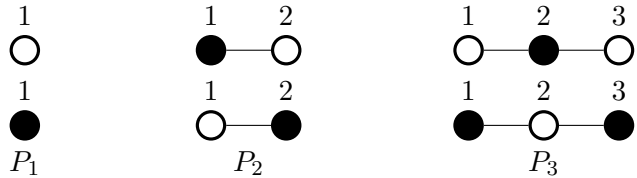

\begin{theorem}
Let $n \in \mathbb{N}.$ Then $\textup{ic}(P_1) = 2$ and $\textup{ic}(P_n) = 2F_{n-1}$ for $n \geq 2.$
\end{theorem}

\begin{theorem}
Let $n \in \mathbb{N}.$ Let $\mathcal{C}$ be a uniform random integrated coloring of $P_n,$ and let $k$ be an integer between $\lceil (n-1)/2 \rceil$ and $n-1.$ Then
\[\textup{Pr}[\textup{mix}(C) = k] = \frac{2\binom{k-1}{n-k-1}}{\textup{ic}(P_n)}.\]
\end{theorem}

\begin{proof}
Given a choice of color for the first vertex, we can model an integrated coloring $\mathcal{C} \in \text{IC}(P_n)$ as a word of length $n - 1$ over the alphabet $\{\textsf{B}, \textsf{U}\}$ consisting of $k$ \textsf{B}s (for balanced edges) and $n - 1 - k$ \textsf{U}s (for unbalanced edges) that avoid the subword \textsf{UU} and cannot begin or end with \textsf{U}. Consequently, the word comprises of $n - k$ nonempty runs of \textsf{B}s separated by the \textsf{U}s. A standard ``stars and bars'' argument shows that there are $\binom{k - 1}{n - k - 1}$ such words, along with two choices for the color of the first vertex.
\end{proof}

Since any integrated coloring of $P_n$ has at least $(n-1)/2$ balanced edges, we arrive at the following formula for $\text{ic}(P_n)$, which can be viewed as a ``reverse rearrangement'' of the familiar binomial identity
\[\sum_{k = 0}^{\left\lfloor \frac{n-1}{2} \right\rfloor} \binom{n - k - 1}{k} = F_{n-1}.\]

\begin{corollary}
For each $n \in \mathbb{N},$
\[\sum_{k=\left\lceil \frac{n - 1}{2} \right\rceil}^{n-1} 2\binom{k-1}{n-k-1} = \textup{ic}(P_n).\]
\end{corollary}

\begin{theorem}
The bivariate generating function for integrated colorings of paths, counted according to the order of the path (indicated by the exponent of the variable $z$) and the mixing number of the coloring (indicated by the exponent of the variable $u$), is
\[F(u, z) = \frac{2z-2uz^3}{1-uz-uz^2}.\]
\end{theorem}

\begin{proof}
Based on Figure \ref{coloringpaths}, we can split up the generating function as
\begin{align}
F(u, z) &= \sum_{n = 1}^{\infty}\sum_{\mathcal{C} \in \text{IC}(P_n)} u^{\text{mix}(\mathcal{C})}z^n
\\ &=2z + 2uz^2 + \sum_{n = 3}^{\infty}\sum_{\substack{\mathcal{C} \in \text{IC}(P_n) \\ \text{balanced}\ (2,3)}} u^{\text{mix}(\mathcal{C})}z^n + \sum_{n = 4}^{\infty}\sum_{\substack{\mathcal{C} \in \text{IC}(P_n) \\ \text{unbalanced}\ (2,3)}} u^{\text{mix}(\mathcal{C})}z^n.
\end{align}
Since any integrated coloring of $P_n$ that renders edge $(2,3)$ balanced can be imagined as a balanced edge followed by an integrated coloring of $P_{n-1},$ it follows that
\begin{equation}
    \sum_{n = 3}^{\infty}\sum_{\substack{\mathcal{C} \in \text{IC}(P_n) \\ \text{balanced}\ (2,3)}} u^{\text{mix}(\mathcal{C})}z^n = uz\sum_{n = 3}^{\infty}\sum_{\mathcal{C} \in \text{IC}(P_{n-1})} u^{\text{mix}(\mathcal{C})}z^{n-1}.
\end{equation}
Similarly, any integrated coloring of $P_n$ that renders edge $(2,3)$ unbalanced can be imagined as a balanced edge and then an unbalanced edge followed by an integrated coloring of $P_{n-2},$ and so
\begin{equation}
\sum_{n = 4}^{\infty}\sum_{\substack{\mathcal{C} \in \text{IC}(P_n) \\ \text{unbalanced}\ (2,3)}} u^{\text{mix}(\mathcal{C})}z^n = uz^2\sum_{n = 4}^{\infty}\sum_{\mathcal{C} \in \text{IC}(P_{n-2})} u^{\text{mix}(\mathcal{C})}z^{n-2}.
\end{equation}
Hence,
\begin{align}
F(u,z) &= 2z + 2uz^2 + uz(F(u,z)-2z) + uz^2(F(u,z) - 2z)
\\ &= 2z - 2uz^3 + uzF(u,z) + uz^2F(u,z).
\end{align}
Solving the resulting functional equation for $F(u,z)$ completes the proof.
\end{proof}

For each $n \in \mathbb{N},$ let $Y_n$ be the random variable corresponding to $\text{mix}(\mathcal{C})$ taken over $\text{IC}(P_n)$ endowed with the uniform probability measure. The probability generating function (PGF) of $Y_n$ is then
\begin{equation}
q_n(u) := \mathbb{E}u^{Y_n} = \frac{[z^n]F(u,z)}{\text{ic}(P_n)}.
\end{equation}
For nonzero $u \in \mathbb{C},$ let
\begin{align*}
    \zeta_1 &:= \zeta_1(u) = \frac{1}{2}\left(\sqrt{1+\frac{4}{u}} - 1\right), &\zeta_2 &:= \zeta_2(u) = \frac{1}{2}\left(\sqrt{1+\frac{4}{u}} + 1\right).
\end{align*}
Finally, let $\phi$ be the golden ratio $\frac{\sqrt{5}+1}{2}.$

\begin{lemma}
There is a real number $\epsilon > 0$ such that
\[q_n(u) = \phi\sqrt{5}\left(\phi\zeta_1\right)^{-n}\frac{\zeta_1\zeta_2(1-u\zeta_1^2)}{\zeta_1 + \zeta_2}\left(1+O\left(\frac{1}{\phi^n}\right)\right)\ \text{as}\ n \to \infty\]
uniformly inside the disc $|u - 1| < \epsilon.$
\end{lemma}

\begin{proof}
Consider the functions
\[f(u) = |\zeta_1(u)| - |\zeta_2(u)|\ \text{and}\ g(u) = \left|\frac{\zeta_2(u)}{\zeta_1(u)+\zeta_2(u)}\right|,\]
which are continuous in their respective domains. Note that $\zeta_2(1) = \phi$ and $\zeta_1(1) = \frac{\sqrt{5}-1}{2} = \frac{1}{\phi}.$ Since $f(1) = -1,$ $f$ is negative throughout some $\epsilon_f$-neighborhood of $u = 1.$ Similarly, since $g(1) = \frac{\phi}{\frac{1}{\phi}+\phi},$ there is an $\epsilon_g$-neighborhood of $u=1$ over which $0 < g(u) < 1.$ Pick $\epsilon = \text{min}\{\epsilon_f, \epsilon_g\}.$

For fixed $u_0 \in \mathbb{C}\backslash\{0\},$ the single-variable function
\begin{equation}
F(u_0,z) = \frac{2z-2u_0z^3}{\left(1+\frac{z}{\zeta_2}\right)\left(1-\frac{z}{\zeta_1}\right)} = \frac{2z\zeta_2(1-u_0z^2)}{z+\zeta_2}\cdot\frac{1}{1-\frac{z}{\zeta_1}}.
\end{equation}
is analytic at $z = 0.$ Moreover, if $|u_0 - 1| < \epsilon,$ then $|\zeta_1| < |\zeta_2|,$ and so $z = \zeta_1$ is the lone dominant singularity of $F(u_0, z)$ in this neighborhood. We can thus apply the Flajolet-Odlyzko transfer theorem (see \cite{FO} and Theorem VI.4 on p.\! 393 of \cite{FS}) to get that
\begin{equation}
\label{asymptotic1}
[z^n]F(u,z) = \zeta_1^{-n}\frac{2\zeta_1\zeta_2(1-u\zeta_1^2)}{\zeta_1 + \zeta_2}\left(1+O\left(g(u)^n\right)\right)\ \text{as}\ n \to \infty,
\end{equation}
provided that $|u - 1| < \epsilon.$ Because $g$ is bounded by 1 over the chosen $\epsilon$-neighborhood, and
\begin{equation}
\text{ic}(P_n) = 2F_{n-1} = \frac{2}{\sqrt{5}}\phi^{n-1}(1 + o(1))\ \text{as}\ n \to \infty,
\end{equation}
dividing both sides of (\ref{asymptotic1}) by the leading term asymptotic of $\text{ic}(P_n)$ begets the desired result.
\end{proof}

We now apply the following ``quasi-powers'' theorem due to Hwang (\cite{Hwang}; also Theorem IX.8 on p. 645 of \cite{FS}) to obtain a central limit theorem for $Y_n$.

\begin{theorem}
\label{quasipowers}
Let $(X_n)_{n = 1}^{\infty}$ be a sequence of nonnegative discrete random variables (supported by a subset of $\mathbb{Z}_{\geq 0}$) with PGFs $(p_n(u))_{n = 1}^{\infty}.$ Assume that, uniformly in a fixed complex neighborhood of $u = 1,$ for some $\beta_n, \kappa_n \to \infty$ as $n \to \infty,$ there holds
\[p_n(u) = A(u)\cdot B(u)^{\beta_n}\left(1 + O\left(\frac{1}{\kappa_n}\right)\right)\ \text{as}\ n \to \infty,\]
where $A(u)$ and $B(u)$ are analytic at $u = 1$ and $A(1) = B(1) = 1.$ Assume finally that $B(u)$ satisfies the variability condition,
\[\mathfrak{v}(B(u)) := B''(1) + B'(1) - B'(1)^2 > 0.\]
Under these conditions, the mean and variance of $X_n$ satisfy
\begin{align*}
\mu_n &:= \mathbb{E}X_n = \beta_nB'(1) + A'(1) + O\left(\frac{1}{\kappa_n}\right), \\
\sigma_n^2 &:= \mathbb{V}X_n = \beta_n\mathfrak{v}(B(u)) + \mathfrak{v}(A(u)) + O\left(\frac{1}{\kappa_n}\right).
\end{align*}
The distribution of $X_n$ is, after standardization, asymptotically Gaussian, and the speed of convergence to the Gaussian limit is $O(\kappa_n^{-1} + \beta_n^{-1/2}):$
\[\mathbb{P}\left[\frac{X_n - \mu_n}{\sigma_n} \leq x\right] = \Phi(x) + O\!\left(\frac{1}{\kappa_n} + \frac{1}{\sqrt{\beta_n}}\right)\ \text{as}\ n \to \infty\]
for all $x \in \mathbb{R},$ where $\Phi(x)$ is the distribution function of a standard normal random variable,
\[\Phi(x) = \frac{1}{\sqrt{2\pi}}\int_{-\infty}^{x} e^{-t^2/2}\,dt.\]
\end{theorem}

\begin{theorem}
For all $x \in \mathbb{R},$
\[\mathbb{P}\left[\frac{Y_n - \mathbb{E}Y_n}{\sqrt{\mathbb{V}Y_n}} \leq x\right] = \Phi(x) + O\!\left(\frac{1}{\sqrt{n}}\right)\ \text{as}\ n \to \infty,\]
where
\[\mathbb{E}Y_n \sim \left(\frac{\sqrt{5}}{10}+\frac{1}{2}\right)n\ \text{and}\ \mathbb{V}Y_n \sim \left(\frac{39\sqrt{5}}{250} + \frac{63}{250}\right)n.\]
\end{theorem}

\begin{proof}
Apply Theorem \ref{quasipowers} to $q_n(u)$ with $A(u) = \frac{\phi\sqrt{5}\zeta_1\zeta_2(1-u\zeta_1^2)}{\zeta_1 + \zeta_2},$ $B(u) = \frac{1}{\phi\zeta_1},$ $\beta_n = n,$ and $\kappa_n = \phi^n.$ Since the leading term asymptotics of the mean and variance are entirely determined by the first two derivatives of $B(u),$ we can forego differentiating $A(u).$ Straightforward, albeit tedious calculus will show that $B'(1) = \frac{\sqrt{5}}{10} + \frac{1}{2},$ $B''(1) = \frac{\sqrt{5}}{25} - \frac{1}{5},$ and $\mathfrak{v}(B(u)) = \frac{39\sqrt{5}}{250} + \frac{63}{250}.$
\end{proof}

\section{Cycles}

We now emulate the combinatorial analysis from the previous section for $\text{IC}(P_n).$ To acquire a recurrence relation and generating function for $\text{ic}(C_n),$ we employ Macauley, McCammond, and Mortveit's \cite{MMM} tiling description for marked cyclic $n$-bit binary strings that avoid runs of length 3. Here, we can treat an integrated coloring of a cycle as a sequence of wire pieces illustrated in Figure \ref{cycletiles} tied together in the shape of a circle.

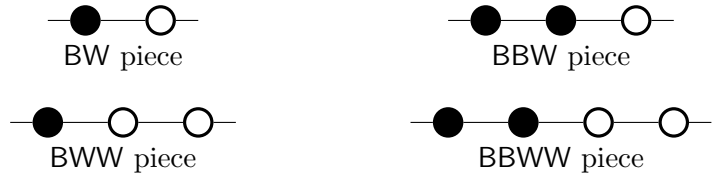
\begin{figure}[h]
\caption{All integrated necklaces can be decomposed into a sequence of the following four wire pieces. Under the conventional labeling, pieces are chained in such way so that the left side is black from the clockwise perspective.}\label{cycletiles}

\bigskip

\begin{minipage}{0.45\textwidth}
\centering
\begin{tikzpicture}[
blacknode/.style={circle, draw=black, fill=black, very thick, minimum size=2},
whitenode/.style={circle, draw=black, fill=white, very thick, minimum size=2},
]

\node[blacknode] (black1) at (0,0){};
\node[whitenode] (white1) at (1,0){};
\draw[-] (black1) -- (white1);
\draw[-] (-0.5,0) -- (black1);
\draw[-] (white1) -- (1.5,0);
\node at (0.5,-0.5) {\textsf{BW} piece};
\end{tikzpicture}
\end{minipage}
\begin{minipage}{0.45\textwidth}
\centering
\begin{tikzpicture}[
blacknode/.style={circle, draw=black, fill=black, very thick, minimum size=2},
whitenode/.style={circle, draw=black, fill=white, very thick, minimum size=2},
]

\node[blacknode] (black1) at (0,0){};
\node[blacknode] (black2) at (1,0){};
\node[whitenode] (white1) at (2,0){};
\draw[-] (black1) -- (black2);
\draw[-] (black2) -- (white1);
\draw[-] (-0.5,0) -- (black1);
\draw[-] (white1) -- (2.5,0);
\node at (1,-0.5) {\textsf{BBW} piece};
\end{tikzpicture}
\end{minipage}

\bigskip

\begin{minipage}{0.45\textwidth}
\centering
\begin{tikzpicture}[
blacknode/.style={circle, draw=black, fill=black, very thick, minimum size=2},
whitenode/.style={circle, draw=black, fill=white, very thick, minimum size=2},
]

\node[blacknode] (black1) at (0,0){};
\node[whitenode] (white1) at (1,0){};
\node[whitenode] (white2) at (2,0){};
\draw[-] (black1) -- (white1);
\draw[-] (white1) -- (white2);
\draw[-] (-0.5,0) -- (black1);
\draw[-] (white2) -- (2.5,0);
\node at (1,-0.5) {\textsf{BWW} piece};
\end{tikzpicture}
\end{minipage}
\begin{minipage}{0.45\textwidth}
\centering
\begin{tikzpicture}[
blacknode/.style={circle, draw=black, fill=black, very thick, minimum size=2},
whitenode/.style={circle, draw=black, fill=white, very thick, minimum size=2},
]

\node[blacknode] (black1) at (0,0){};
\node[blacknode] (black2) at (1,0){};
\node[whitenode] (white1) at (2,0){};
\node[whitenode] (white2) at (3,0){};
\draw[-] (black1) -- (black2);
\draw[-] (black2) -- (white1);
\draw[-] (white1) -- (white2);
\draw[-] (-0.5,0) -- (black1);
\draw[-] (white2) -- (3.5,0);
\node at (1.5,-0.5) {\textsf{BBWW} piece};
\end{tikzpicture}
\end{minipage}
\end{figure}

We will henceforth refer to an integrated coloring of $C_n$ as an \textit{integrated n-necklace}. (For $C_2,$ we mean the multigraph of two vertices connected by two parallel edges.) An integrated $n$-necklace is \textit{out-of-phase} when a single wire piece covers vertices 1 and $n$ and \textit{in-phase} otherwise. We define the \textit{first piece} of the necklace as the wire piece that covers vertex 1. For necklaces with at least two pieces, the \textit{last piece} of the necklace is the one that precedes the first piece in the clockwise direction.

\begin{theorem}
Let $n \in \mathbb{N},$ $n \geq 2.$ Then $\textup{ic}(C_2) = 2,$ $\textup{ic}(C_3) = 6,$ $\textup{ic}(C_4) = 6,$ $\textup{ic}(C_5) = 10,$ and $\textup{ic}(C_n) = \textup{ic}(C_{n-2})+2\textup{ic}(C_{n-3})+ \textup{ic}(C_{n-4})$ for $n \geq 6.$
\end{theorem}

\begin{proof}
Since it is the first piece that determines the phase of a necklace, there are $\text{ic}(C_{n-2})$ integrated $n$-necklaces that end with a $\textsf{BW}$ piece, $\text{ic}(C_{n-3})$ that end with a \textsf{BBW} piece, $\text{ic}(C_{n-3})$ that end with a \textsf{BWW} piece, and $\text{ic}(C_{n-4})$ that end with a \textsf{BBWW} piece. Thus, for $n \geq 6,$
\begin{equation}
\label{cyclerecurrence}
\text{ic}(C_n) = \text{ic}(C_{n-2})+2\text{ic}(C_{n-3})+ \text{ic}(C_{n-4}).
\end{equation}
The initial conditions of (\ref{cyclerecurrence}) are enumerated as follows.
\begin{itemize}
\item $\text{ic}(C_2) = 2.$ Glue the ends of a \textsf{BW} piece. The resulting necklace is equally likely to be in-phase or out-of-phase.

\item $\text{ic}(C_3) = 6.$ Glue the ends of a \textsf{BBW} piece or a \textsf{BWW piece}. In both instances, there are three options for which node is vertex 1.

\item $\text{ic}(C_4) = 6.$ There are two possible linkages.
\begin{enumerate}[i.]
    \item Connect two \textsf{BW} pieces. The resulting necklace is equally likely to be in-phase or out-of-phase.
    \item Glue the ends of a \textsf{BBWW} piece. There are four options for which node is vertex 1.
\end{enumerate}

\item $\text{ic}(C_5) = 10.$ Connect a \textsf{BW} piece with either a \textsf{BBW} or $\textsf{BWW}$ piece.
\begin{enumerate}[i.]
    \item If the first piece is the \textsf{BW} piece, then there are two options for the last piece and two options for the phase. Thus, there are four integrated 5-necklaces that belong to this case.
    \item If the first piece is a \textsf{BBW} piece or a \textsf{BWW} piece. In either instance, there are three options for which node is vertex 1.
\end{enumerate}
\end{itemize}
\end{proof}

The values of $\text{ic}(C_n)$ appear as the sequence A007040 in the OEIS \cite{OEIS} and have been studied in a number of different contexts. A concise formula is $\text{ic}(C_n) = L_n + 2\cos(2n\pi/3),$ where $L_n$ is the $n^{\text{th}}$ Lucas number. The appearance of the Lucas numbers is not surprising, as our interpretation of $\text{IC}(C_n)$ as the set of integrated $n$-necklaces is analogous to the concept of $n$-bracelets, which is thoroughly discussed in \cite{BQ}.

Next, observe that every integrated $n$-necklace has an even mixing number between $n/2$ and $n.$ One way to see this is to note that each wire piece in Figure \ref{cycletiles} contributes exactly two balanced edges to the right of its first node, namely an intermediate edge connecting black and white nodes and then the right end of the piece that attaches its white end node to the black one at the start of the next piece (or itself, as the case may be). Another reason is that balanced edges flip the color of the preceding vertex, and so an even mixing number is required to preserve the cyclicity of the necklace.

\begin{theorem}
Let $n \in \mathbb{N},$ $n \geq 2.$ Let $\mathcal{C}$ be a uniform random integrated $n$-necklace, and let $k$ be an integer between $\lceil n/4 \rceil$ and $\lfloor n/2 \rfloor.$ Then
\[\textup{Pr}[\textup{mix}(\mathcal{C}) = 2k] = \frac{2\binom{2k-1}{n-2k} + 4\binom{2k-1}{n-2k-1}}{\textup{ic}(C_n)}.\]
\end{theorem}

\begin{proof}
Just as with paths, model an integrated $n$-necklace as a word of length $n$ over the alphabet $\{\textsf{B}, \textsf{U}\},$ except now with the following constraints:
\begin{itemize}
    \item avoids the subword \textsf{UU};
    \item cannot both begin and end with \textsf{U};
    \item needs an even number of \textsf{B}s.
\end{itemize}
Now consider a word of length $n$ consisting of $2k$ \textsf{B}s and $n - 2k$ \textsf{U}s that satisfies the constraints aforementioned. Such a word can either be bookended by \textsf{B}s or begin and end with different letters. In the first case, the word comprises of $n - 2k + 1$ nonempty runs of \textsf{B}s separated by the \textsf{U}s. A standard ``stars and bars'' argument shows that there are $\binom{2k - 1}{n - 2k}$ such words, along with two choices for the color of the first vertex. In the second case, the word comprises of $n - 2k$ nonempty runs of \textsf{B}s separated by the \textsf{U}s. There are $\binom{2k - 1}{n - 2k - 1}$ such words, then two choices as to whether to begin with a \textsf{U} or a \textsf{B}, and finally two choices for the color of the first vertex.
\end{proof}

\begin{corollary}
For each integer $n \geq 2,$
\[\sum\limits_{k = \lceil n/4 \rceil}^{\lfloor n/2 \rfloor} \left(2\binom{2k-1}{n-2k} + 4\binom{2k-1}{n-2k-1}\right) = \textup{ic}(C_n).\]
\end{corollary}

\begin{theorem}
The bivariate generating function for integrated necklaces, counted according to the order of the cycle (indicated by the exponent of the variable $z$) and the mixing number of the coloring (indicated by the exponent of the variable $u$), is
\[G(u, z) = \frac{2u^2z^2+6u^2z^3+4u^2z^4}{1-u^2z^2-2u^2z^3-u^2z^4}.\]
\end{theorem}

\begin{proof}
Based on Figure \ref{cycletiles}, we can split up the generating function as
\begin{align}
G(u, z) &= \sum_{n = 1}^{\infty}\sum_{\mathcal{C} \in \text{IC}(C_n)} u^{\text{mix}(\mathcal{C})}z^n
\\ &=2u^2z^2 + 6u^2z^3 + 4u^2z^4 + \sum_{n = 4}^{\infty}\sum_{\substack{\mathcal{C} \in \text{IC}(C_n) \\ \textsf{BW}\ \text{last piece}}} u^{\text{mix}(\mathcal{C})}z^n \nonumber \\
 &\hspace{1em} + \sum_{n = 5}^{\infty}\sum_{\substack{\mathcal{C} \in \text{IC}(C_n) \\ \textsf{BBW}\backslash\textsf{BWW}\ \text{last piece}}} u^{\text{mix}(\mathcal{C})}z^n + \sum_{n = 6}^{\infty}\sum_{\substack{\mathcal{C} \in \text{IC}(C_n) \\ \textsf{BBWW}\ \text{last piece}}} u^{\text{mix}(\mathcal{C})}z^n.
\end{align}
If the last piece of an integrated $n$-necklace is a \textsf{BW} piece, then its removal yields an integrated $(n-2)$-necklace with two fewer balanced edges. Hence
\begin{equation}
    \sum_{n = 4}^{\infty}\sum_{\substack{\mathcal{C} \in \text{IC}(C_n) \\ \textsf{BW}\ \text{last piece}}} u^{\text{mix}(\mathcal{C})}z^n = u^2z^2\sum_{n = 4}^{\infty}\sum_{\mathcal{C} \in \text{IC}(C_{n-2})} u^{\text{mix}(\mathcal{C})}z^{n-2}.
\end{equation}
Repeat this logic for the other possibilities for the last piece to see that
\begin{equation}
\sum_{n = 5}^{\infty}\sum_{\substack{\mathcal{C} \in \text{IC}(C_n) \\ \textsf{BBW}\backslash\textsf{BWW}\ \text{last piece}}} u^{\text{mix}(\mathcal{C})}z^n = 2u^2z^3\sum_{n = 2}^{\infty}\sum_{\mathcal{C} \in \text{IC}(C_{n-3})} u^{\text{mix}(\mathcal{C})}z^{n-3}
\end{equation}
and
\begin{equation}
    \sum_{n = 6}^{\infty}\sum_{\substack{\mathcal{C} \in \text{IC}(C_n) \\ \textsf{BBWW}\ \text{last piece}}} u^{\text{mix}(\mathcal{C})}z^n = u^2z^4\sum_{n = 6}^{\infty}\sum_{\mathcal{C} \in \text{IC}(C_{n-4})} u^{\text{mix}(\mathcal{C})}z^{n-4}.
\end{equation}
Hence,
\begin{equation}
G(u,z) = 2u^2z^2 + 6u^2z^3 + 4u^2z^4 + u^2z^2G(u,z)+2u^2z^3G(u,z)+u^2z^4G(u,z).
\end{equation}
Solving the resulting functional equation for $G(u,z)$ completes the proof.
\end{proof}

For each $n \in \mathbb{N},$ let $Z_n$ be the random variable corresponding to $\text{mix}(\mathcal{C})$ taken over $\text{IC}(C_n)$ endowed with the uniform probability measure. The PGF of $Z_n$ is then
\begin{equation}
r_n(u) := \mathbb{E}u^{Z_n} = \frac{[z^n]G(u,z)}{\text{ic}(C_n)}.
\end{equation}

In what follows, we will reprise the notation of Lemma 7 and its proof. We also introduce two new functions, defined for nonzero $u \in \C:$
\begin{align*}
\zeta_3 &:= \zeta_3(u) = \frac{1}{2}\left(\sqrt{1-\frac{4}{u}} - 1\right), &\zeta_4 &:= \zeta_4(u) = \frac{1}{2}\left(\sqrt{1-\frac{4}{u}} + 1\right).
\end{align*}

\begin{lemma}
There is a real number $\epsilon > 0$ such that
\[r_n(u) = (\phi\zeta_1)^{-n}\frac{2u^2\zeta_1^2\zeta_2(1+\zeta_1)(1+2\zeta_1)}{(\zeta_1+\zeta_2)(1+u\zeta_1+u\zeta_1^2)}\left(1+O\left(\frac{1}{\phi^n}\right)\right)\ \text{as}\ n \to \infty\]
uniformly inside the disc $|u-1|<\epsilon.$
\end{lemma}

\begin{proof}
Consider the continuous functions $h_1, h_2 : \mathbb{C}\backslash\{0\} \rightarrow \mathbb{R}$ given by
\[h(u) = |\zeta_1(u)| - |\zeta_3(u)|\ \text{and}\ h^*(u) = |\zeta_1(u)| - |\zeta_4(u)|.\]
Note that $|\zeta_3(1)| = |\zeta_4(1)| = 1 > \frac{1}{\phi}.$ Therefore, $h(1) = h^*(1) < 0,$ and so $h$ and $h^*$ are negative throughout some $\epsilon_h$-neighborhood and $\epsilon_{h^*}$-neighborhood, respectively, of $u = 1.$ Furthermore,
\begin{equation}
\frac{g(1)}{1+\zeta_1(1)+\zeta_1(1)^2} = \frac{\phi}{2(\frac{1}{\phi}+\phi)},
\end{equation}
so there is an $\epsilon_{g^*}$-neighborhood of $u = 1$ in which $0 < \frac{g(u)}{|1+u\zeta_1(u)+u\zeta_1(u)^2|}<1.$ Pick $\epsilon = \text{min}\{\epsilon_f, \epsilon_{g^*}, \epsilon_h, \epsilon_{h^*}\}.$

For fixed $u_0 \in \mathbb{C}\backslash\{0\},$ the single-variable function
\[G(u_0, z) = \frac{2u_0^2z^2+6u_0^2z^3+4u_0^2z^4}{(1 - \frac{z}{\zeta_1})(1 + \frac{z}{\zeta_2})(1 - \frac{z}{\zeta_3})(1 + \frac{z}{\zeta_4})} = \frac{2u_0^2z^2\zeta_2(1+z)(1+2z)}{(z+\zeta_2)(1+u_0z+u_0z^2)}\cdot\frac{1}{1-\frac{z}{\zeta_1}}\]
is analytic at $z = 0.$ Moreover, if $|u_0 - 1| < \epsilon,$ then $|\zeta_1|$ is smaller than each of $|\zeta_2|,$ $|\zeta_3|,$ and $|\zeta_4|.$ So $z = \zeta_1$ is the lone dominant singularity of $G(u_0, z)$ in this neighborhood. We can thus apply the Flajolet-Odlyzko transfer theorem to get that
\begin{equation}
\label{asymptotic2}
[z^n]G(u,z) = \zeta_1^{-n}\frac{2u^2\zeta_1^2\zeta_2(1+\zeta_1)(1+2\zeta_1)}{(\zeta_1+\zeta_2)(1+u\zeta_1+u\zeta_1^2)}\left(1+O\left(\left(\frac{g(u)}{|1+u\zeta_1+u\zeta_1^2|}\right)^n\right)\right)
\end{equation}
as $n \to \infty,$ provided that $|u - 1| < \epsilon.$ Because $\frac{g(u)}{|1+u\zeta_1+u\zeta_1^2|}$ is bounded by 1 over the chosen $\epsilon$-neighborhood, and
\begin{equation}
\text{ic}(C_n) = L_n + 2\cos(2n\pi/3) = \phi^n(1 + o(1))\ \text{as}\ n \to \infty,
\end{equation}
dividing both sides of (\ref{asymptotic2}) by the leading term asymptotic of $\text{ic}(C_n)$ begets the desired result.
\end{proof}

Notice that $F(u,z)$ and $G(u,z)$ share the same singular factor $(1 - z/\zeta_1)^{-1}$ near $u=1.$ Furthermore, the sample space cardinalities $\text{ic}(P_n)$ and $\text{ic}(C_n)$ exhibit the same exponential growth rate of $\phi^n.$ As a result, the distributions of the standardized $Y_n$ and $Z_n$ are asymptotically equivalent.

\begin{theorem}
For all $x \in \mathbb{R},$
\[\mathbb{P}\left[\frac{Z_n - \mathbb{E}Z_n}{\sqrt{\mathbb{V}Z_n}} \leq x\right] = \Phi(x) + O\!\left(\frac{1}{\sqrt{n}}\right)\ \text{as}\ n \to \infty,\]
where
\[\mathbb{E}Z_n \sim \left(\frac{\sqrt{5}}{10}+\frac{1}{2}\right)n\ \text{and}\ \mathbb{V}Z_n \sim \left(\frac{39\sqrt{5}}{250} + \frac{63}{250}\right)n.\]
\end{theorem}

\section{A General Upper Bound for $\textup{ic}(G)$}

In this section, we invoke the second-moment method to calculate an upper bound for $\text{ic}(G)$ that applies to any simple graph with maximum degree $\Delta(G) \geq 2$. Of course, $\Delta(G) \leq 2$ if and only if $G$ is the union of disconnected edges and isolated vertices. If such a graph $G$ has $\ell$ isolated vertices, then $\text{ic}(G) = 2^{|E|+\ell} = 2^{|V|-|E|}.$

Recall that a vertex is called a pendant if its degree is 1. Clearly, an integrated pendant vertex must be the opposite color of its neighbor, and isolated vertices are vacuously integrated (a perhaps unintended consequence of the definition). So our strategy will be to use the Chebyshev-Cantelli inequality, which states that for any random variable $X$ and real $a > 0,$
\[\text{Pr}[X-\mathbb{E}X \geq a] \leq \frac{\mathbb{V}X}{\mathbb{V}X+a^2},
\]
to bound the probability that, in a semi-random coloring, at least one non-pendant vertex of degree 2 or higher is not integrated. To help set up Theorem \ref{probabilisticmethod} along with its proof and corollaries, we introduce the following notation.

\begin{itemize}
\item $V':$ the set of all non-pendant vertices in $V.$
\item $\lambda_v:$ the number of vertices in $V'$ that are adjacent to $v.$
\item $\lambda_{v, w}:$ the mutual degree of $v$ and $w,$ that is, the number of neighbors common to $v$ and $w.$
\item $V'':$ the set of all vertices $v \in V'$ such that $\lambda_v > \frac{1}{2}\text{deg}(v).$
\item $\alpha_{i,j}(v, w):$ the following sum defined for pairs of vertices $v$ and $w,$
\[\sum_{\substack{a,b,c \in \mathbb{Z}_{\geq 0} \\ a + b \geq \lambda_v-\frac{1}{2}\text{deg}(v)-ij \\ i\lambda_{v,w}+(-1)^{i}a + c \geq \lambda_w - \frac{1}{2}\text{deg}(w)-ij}} \frac{\binom{\lambda_{v,w}}{a}\binom{\lambda_v-\lambda_{v,w}-j}{b}\binom{\lambda_w-\lambda_{v,w}-j}{c}}{2^{\lambda_v+\lambda_w-\lambda_{v,w}-2j}}.\]
In the proof of Theorem \ref{probabilisticmethod}, we can think of the indices $i$ and $j$ in $\alpha_{i,j}$ acting as boolean operators for the events of $v$ and $w$ being different colors (indicated by the binary value of $i$) and being adjacent (indicated by the binary value of $j$).
\item $\chi := \chi_{2\Z}$ is the characteristic function for the set of even integers.
\item $R := 1+\binom{r}{\frac{r}{2}}2^{-r}.$
\end{itemize}

\begin{theorem}
\label{probabilisticmethod}
Let $G$ be a simple graph with $\Delta(G) \geq 2.$ Then
\[\textup{ic}(G) \leq \frac{\sigma^2}{\sigma^2+(|V''|-\mu)^2}\cdot2^{|V'|},\]
where
\[\mu = \sum_{v \in V''} \sum_{k = \lceil\lambda_v-\frac{1}{2}\textup{deg}(v)\rceil}^{\lambda_v} \binom{\lambda_v}{k}2^{-\lambda_v}\]
and
\begin{align*}
\sigma^2 &= \mu + \sum_{\substack{\textup{distinct}\ v, w \in V'' \\ v,w\ \textup{adjacent}}}(\alpha_{0,1}(v,w) + \alpha_{1,1}(v,w)) 
\\ & \hspace{2em}+ \sum_{\substack{\textup{distinct}\ v, w \in V'' \\ v,w\ \textup{not adjacent}}} (\alpha_{0,0}(v,w) + \alpha_{1,0}(v,w)) - \mu^2.
\end{align*}
\end{theorem}

\begin{proof}
Consider a semi-random coloring $\mathcal{C}$ of $G$ obtained by coloring each vertex in $V'$ independently with the choice of black or white being equally likely, then assign any pendant vertices the opposite color of their neighbors. Observe that any vertex in $V\backslash V''$ is then automatically integrated. For each $v \in V'',$ let $I_v$ be the indicator random variable for the event that $v$ is integrated. If $X =\ \sum_{v \in V''} I_v,$ then by linearity of expectation along with the fact that for every $v \in V'',$ the vertex $v$ is guaranteed at least $\text{deg}(v) - \lambda_v$ neighbors of the opposite color,
\begin{equation}
\mathbb{E}X = \sum_{v \in V''} \mathbb{E}I_v = \sum_{v \in V''} \text{Pr}\!\left[\text{mix}(v) \geq \frac{1}{2}\text{deg}(v)\right] = \mu,
\end{equation}
and
\begin{align}
\mathbb{E}[X^2] &= \mu + \sum_{\text{distinct}\ v, w\in V''} 2\mathbb{E}[I_vI_w]
\\ &= \mu + \sum_{\substack{\text{distinct}\ v, w \in V'' \\ v,w\ \text{adjacent}}} 2\mathbb{E}[I_vI_w] + \sum_{\substack{\text{distinct}\ v, w \in V'' \\ v,w\ \text{not adjacent}}} 2\mathbb{E}[I_vI_w].
\end{align}
Now for each pair of distinct vertices $v, w \in V'',$ let $J_{v,w}$ be the indicator random variable for the event that $v$ and $w$ are different colors. It is easy to see that $J_{v,w}$ amounts to a fair Bernoulli trial, and
\begin{align}
\mathbb{E}[I_vI_w] &= \text{Pr}[I_vI_w = 1]
\\ &= \text{Pr}[I_vI_w = 1\ \text{and}\ J_{v,w}=0] + \text{Pr}[I_vI_w = 1\ \text{and}\ J_{v,w}=1].
\end{align}
If $v$ and $w$ are adjacent, then
\[\mathbb{E}[I_vI_w] = \frac{1}{2}\alpha_{0,1}(v,w) + \frac{1}{2}\alpha_{1,1}(v,w).\]
On the other hand, if $v$ and $w$ are not adjacent, then
\[\mathbb{E}[I_vI_w] = \frac{1}{2}\alpha_{0,0}(v,w) + \frac{1}{2}\alpha_{1,0}(v,w).\]
Piecing everything together, it follows that
\begin{align}
\mathbb{E}[X^2] &= \mu + \sum_{\substack{\textup{distinct}\ v, w \in V'' \\ v,w\ \textup{adjacent}}}(\alpha_{0,1}(v,w) + \alpha_{1,1}(v,w)) \nonumber
\\ & \hspace{2em}+ \sum_{\substack{\textup{distinct}\ v, w \in V'' \\ v,w\ \textup{not adjacent}}} (\alpha_{0,0}(v,w) + \alpha_{1,0}(v,w))
\end{align}
and so $\mathbb{V}X = \sigma^2.$ Therefore, by the Chebyshev-Cantelli inequality,
\begin{align}
\text{Pr}[\mathcal{C}\ \text{is integrated}] &= \text{Pr}[X = |V''|]
\\ &= \text{Pr}[X - \mu \geq |V''| - \mu] \leq \frac{\sigma^2}{\sigma^2+(|V''|-\mu)^2}.
\end{align}
Multiply the above inequality by the total number of semi-random colorings to complete the proof.
\end{proof}

\begin{corollary}
\label{mindegree}
Let $G$ be a simple graph with minimum degree $\delta(G) \geq 2.$ Then
\[\textup{ic}(G) \leq \frac{4\sigma^2}{4\sigma^2+(|V|+2\mu')^2}\cdot2^{|V|},\]
where
\[\mu' = \sum_{\substack{v \in V\\ \deg(v)\ \textup{is even}}} \binom{\deg(v)}{\frac{\deg(v)}{2}}2^{-\deg(v)-1}\]
and
\begin{align*}
\sigma^2 &= \frac{|V|}{2}+\mu' + \sum_{\substack{\textup{distinct}\ v, w \in V \\ v,w\ \textup{adjacent}}}(\alpha_{0,1}(v,w) + \alpha_{1,1}(v,w)) 
\\ & \hspace{2em}+ \sum_{\substack{\textup{distinct}\ v, w \in V \\ v,w\ \textup{not adjacent}}} (\alpha_{0,0}(v,w) + \alpha_{1,0}(v,w)) - \left(\frac{|V|}{2}+\mu'\right)^2.
\end{align*}
\end{corollary}

\begin{proof}
If $\delta(G) \geq 2,$ then $G$ is pendant-free. So $V=V'=V'',$ and $\lambda_v = \deg(v)$ for all $v \in V.$ Due to the symmetry of the binomial coefficients,
\begin{align}
\mu &= \sum_{v \in V} \sum_{k = \lceil\frac{1}{2}\textup{deg}(v)\rceil}^{\deg(v)} \binom{\deg(v)}{k}2^{-\deg(v)}
\\ &= \sum_{\substack{v \in V \\ \text{deg}(v)\ \text{is odd}}} \sum_{k = \lceil\frac{1}{2}\textup{deg}(v)\rceil}^{\deg(v)} \binom{\deg(v)}{k}2^{-\deg(v)} \nonumber \\ &\hspace{2.5em}+ \sum_{\substack{v \in V \\ \text{deg}(v)\ \text{is even}}} \sum_{k = \lceil\frac{1}{2}\textup{deg}(v)\rceil}^{\deg(v)} \binom{\deg(v)}{k}2^{-\deg(v)}
\\ &= \sum_{\substack{v \in V \\ \text{deg}(v)\ \text{is odd}}} \frac{1}{2}\cdot2^{\deg(v)}2^{-\deg(v)} \nonumber \\ &\hspace{2.5em}+ \sum_{\substack{v \in V \\ \text{deg}(v)\ \text{is even}}} \frac{1}{2}\left(2^{\deg(v)}+\binom{\deg(v)}{\frac{\deg(v)}{2}}\right)2^{-\deg(v)}
\\ &= \sum_{\substack{v \in V \\ \text{deg}(v)\ \text{is odd}}} \frac{1}{2} + \sum_{\substack{v \in V \\ \text{deg}(v)\ \text{is even}}} \left(\frac{1}{2} + \binom{\deg(v)}{\frac{\deg(v)}{2}}2^{-\deg(v)-1}\right)
\\ &= \frac{1}{2}|V| + \mu'.
\end{align}
Plugging this into Theorem \ref{probabilisticmethod} and then performing a little bit of algebra leads to the desired conclusion.
\end{proof}

Clearly, $\mu'=0$ whenever $G$ only has odd-degree vertices, in which case the upper bound of Corollary \ref{mindegree} simplifies to the rather elegant looking
\[\textup{ic}(G) \leq \frac{4\sigma^2}{4\sigma^2+|V|^2}\cdot2^{|V|}.\]
However, if $G$ is an $r$-regular graph of even valency, then $\mu' = |V|\binom{r}{\frac{r}{2}}2^{-r-1}.$ These observations can be consolidated into the following.

\begin{corollary}
Let $G$ be an $r$-regular graph with $r \geq 2$. Then
\[\textup{ic}(G) \leq \frac{4\sigma^2}{4\sigma^2+R^{2\chi(r)}|V|^2}\cdot2^{|V|},\]
where
\begin{align*}
\sigma^2 &= \frac{1}{2}R^{\chi(r)}|V| + \sum_{\substack{\textup{distinct}\ v, w \in V \\ v,w\ \textup{adjacent}}}(\alpha_{0,1}(v,w) + \alpha_{1,1}(v,w)) 
\\ & \hspace{2em}+ \sum_{\substack{\textup{distinct}\ v, w \in V \\ v,w\ \textup{not adjacent}}} (\alpha_{0,0}(v,w) + \alpha_{1,0}(v,w)) - \frac{1}{4}R^{2\chi(r)}|V|^2.
\end{align*}
\end{corollary}

Note that if $G$ is $r$-regular, $r \geq 2$, then
\begin{equation}
\alpha_{i,j}(v,w) = \sum_{\substack{a,b,c \in \mathbb{Z}_{\geq 0} \\ a + b \geq \frac{r}{2}-ij \\ i\lambda_{v,w}+(-1)^{i}a + c \geq \frac{r}{2}-ij}} \frac{\binom{\lambda_{v,w}}{a}\binom{r-\lambda_{v,w}-j}{b}\binom{r-\lambda_{v,w}-j}{c}}{2^{2r-\lambda_{v,w}-2j}}.
\end{equation}
If, in addition, $G$ is strongly regular with parameters $(n,r,\lambda,\lambda'),$ or $\text{srg}(n,r,\lambda,\lambda')$ for short, then
\begin{equation}
\alpha_{i,0} :=\alpha_{i,0}(v, w) = \sum_{\substack{a,b,c \in \mathbb{Z}_{\geq 0} \\ a + b \geq \frac{r}{2} \\ i\lambda'+(-1)^{i}a + c \geq \frac{r}{2}}} \frac{\binom{\lambda'}{a}\binom{r-\lambda'}{b}\binom{r-\lambda'}{c}}{2^{2r-\lambda'}}
\end{equation}
for all non-adjacent $v$ and $w,$ and
\begin{equation}
\alpha_{i,1} := \alpha_{i,1}(v, w) = \sum_{\substack{a,b,c \in \mathbb{Z}_{\geq 0} \\ a + b \geq \frac{r}{2}-i \\ i\lambda+(-1)^{i}a + c \geq \frac{r}{2}-i}} \frac{\binom{\lambda}{a}\binom{r-\lambda-1}{b}\binom{r-\lambda-1}{c}}{2^{2r-\lambda-2}}
\end{equation}
for all adjacent $v$ and $w.$ Since in an $r$-regular graph there are $\frac{1}{2}nr$ pairs of adjacent vertices and $\frac{1}{2}n(n-r-1)$ pairs of non-adjacent vertices, we arrive at the following.

\begin{corollary}
Let $G$ be an $\textup{srg}(n,r,\lambda,\lambda')$ graph. Then
\[\textup{ic}(G) \leq \frac{4\sigma^2}{4\sigma^2+R^{2\chi(r)}n^2}\cdot2^{n},\]
where
\[\sigma^2 = \frac{1}{2}R^{\chi(r)}n + \frac{1}{2}nr(\alpha_{0,1} + \alpha_{1,1}) + \frac{1}{2}n(n-r-1)(\alpha_{0,0} + \alpha_{1,0}) - \frac{1}{4}R^{2\chi(r)}n^2.\]
\end{corollary}

\section*{Acknowledgments}

The authors' collaboration was supported in part by a Title III grant from Xavier University of Louisiana. The student co-authors thank XULA's Center for Undergraduate Research and Graduate Opportunity for their support. The authors also thank Patrick Vernon for many helpful discussions about this project.


\bibliographystyle{plain}

\end{document}